\documentclass[10pt]{article}

\usepackage[utf8]{inputenc} 

\usepackage{geometry} 
\usepackage{graphicx} 

\usepackage[parfill]{parskip} 

\usepackage{booktabs} 
\usepackage{array} 
\usepackage{paralist} 
\usepackage{verbatim} 
\usepackage{subfig} 

\usepackage{fancyhdr} 
\usepackage[nottoc,notlof,notlot]{tocbibind} 
\usepackage[titles,subfigure]{tocloft} 

\usepackage[T1]{fontenc}
\usepackage[french]{babel}
\usepackage{amsmath,bm}
\usepackage{amsfonts}
\usepackage{amsthm}
\usepackage{enumitem}

\begingroup
    \makeatletter
    \@for\theoremstyle:=definition,remark,plain\do{%
        \expandafter\g@addto@macro\csname th@\theoremstyle\endcsname{%
            \addtolength\thm@preskip\parskip
            }%
        }
\endgroup

\usepackage{mathtools}
\usepackage{tikz-cd}
\usepackage{amssymb}
\usepackage{ytableau}
\usepackage{extpfeil}
\usepackage{bm}
\usepackage{xcolor}
\usepackage{mathrsfs}
\usepackage{todonotes}

\definecolor{green}{rgb}{0.,0.6,0.}

\usepackage{setspace}
\title{A new order on integer partitions}
\author{Étienne Tétreault}
\date{} 

\begin{document}
	\maketitle
	\pagenumbering{gobble}
	\pagenumbering{arabic}
	\newcommand{\s}{\mathfrak{S}}
	\newcommand{\infl}{\text{Inf}}
	\newcommand{\mwrn}{\s_m \wr \s_n}
	\newcommand{\tbduu}[2]{\ytableausetup{boxsize=1.3em, aligntableaux=center, tabloids} \begin{ytableau} #2 \\ #1 \end{ytableau}}
	\newcommand{\tyd}[2]{\ytableausetup{boxsize=1.3em} \begin{ytableau} #1 & #2 \end{ytableau}}
	\newcommand{\tyt}[3]{\ytableausetup{boxsize=1.3em} \begin{ytableau} #1 & #2 & #3 \end{ytableau}}
	\newcommand{\tbdd}[4]{\ytableausetup{boxsize=1.3em} \begin{ytableau} #3 & #4 \\ #1 & #2 \end{ytableau}}
	\newcommand{\ttbdd}[4]{\ytableausetup{tabloids, boxsize=1.9em} \begin{ytableau} #3 & #4 \\ #1 & #2 \end{ytableau}}
	\newcommand{\tyduu}[4]{\ytableausetup{tabloids, boxsize=1.9em} \begin{ytableau} #4 \\ #3 \\ #1 & #2 \end{ytableau}}
	\newcommand{\tytu}[4]{\ytableausetup{boxsize=1.3em} \begin{ytableau} #4 \\ #1 & #2 & #3 \end{ytableau}}
	\newcommand{\tytd}[5]{\ytableausetup{boxsize=1.3em} \begin{ytableau} #4 & #5 \\ #1 & #2 & #3 \end{ytableau}}
	\newcommand{\tydu}[3]{\ytableausetup{boxsize=1.5em, aligntableaux=center} \begin{ytableau} #3 \\ #1 & #2 \end{ytableau}}
	\newcommand{\tbddu}[3]{\ytableausetup{boxsize=1.5em, aligntableaux=center, tabloids} \begin{ytableau} #3 \\ #1 & #2 \end{ytableau}}
	\newcommand{\tbddd}[4]{\ytableausetup{boxsize=1.5em, aligntableaux=center, tabloids} \begin{ytableau} #3 & #4 \\ #1 & #2 \end{ytableau}}
	\newcommand{\tbdtd}[5]{\ytableausetup{boxsize=1.5em, aligntableaux=center, tabloids} \begin{ytableau} #4 & #5 \\ #1 & #2 & #3 \end{ytableau}}
	\newcommand{\ltbdd}[4]{\ytableausetup{boxsize=1.9em} \begin{ytableau} #3 && #4 \\ #1 && #2 \end{ytableau}}
	\newcommand{\hooklongrightarrow}{\lhook\joinrel\longrightarrow}
	\newcommand{\HOM}{\text{Hom}}
	\newcommand{\SSYT}{\text{SSYT}}
	\newcommand{\sgn}{\text{sgn}}
	\newcommand{\C}{\mathbb{C}}
	\newcommand{\R}{\mathbb{R}}
	\newcommand{\Q}{\mathbb{Q}}
	\newcommand{\Z}{\mathbb{Z}}
	\newcommand{\N}{\mathbb{N}}
	\newcommand{\Star}[2]{\overset{#2}{\underset{#1}{*}}}
	\newtheorem{thm}{Theorem}[section]
	\newtheorem{prop}[thm]{Proposition}
	\newtheorem{cor}[thm]{Corollary}
	\newtheorem{conj}[thm]{Conjecture}
	\newtheorem{lem}[thm]{Lemma}
	\renewcommand{\proofname}{Proof}
	\renewcommand{\abstractname}{Abstract}
	\renewcommand{\bibname}{References}

\begin {abstract}
Considering Schur positivity of differences of plethysms of homogeneous symmetric functions, we introduce a new relation on integer partitions. This relation is conjectured to be a partial order, with its restriction to one part partitions equivalent to the classical Foulkes conjecture. We establish some of the properties of this relation via the construction of explicit inclusion of modules whose characters correspond to the plethysms considered. We also prove some stability properties for the number of irreducible occurring in these modules as $m$ grows.
\end {abstract}

It’s been more than eighty years since Littlewood \cite{Littlewood} introduced the plethysm operation on symmetric functions (although the name was only introduced in 1950). This binary operation, denoted by $f[g]$, plays a fundamental role in representation theory, and its calculation raises many interesting questions. Indeed, calculating the “structure” coefficients of plethysms of the Schur functions is considered by Stanley \cite{Stanley99} as a key problem in algebraic combinatorics, and it also appears at the forefront of current research in Geometric Complexity Theory. A longstanding conjecture, stated by Foulkes \cite{Foulkes} in 1953, simply concerns inequalities between some of these coefficients in simple cases. With the aim of setting up our own context, we cast the statement of his conjecture in terms of coefficients $a_{\nu,\mu}^{\lambda}$ of the Schur expansion of the plethysm
\[
h_{\nu}[h_{\mu}] = \sum_\lambda a_{\nu[\mu]}^{\lambda} \, s_\lambda,
\]
of complete homogeneous symmetric functions, which are well known to be positive integers.
Foulkes’ conjecture states that for all $n \leq m$, and all partition $\lambda$ of $mn$, we have $a_{(n)[(m)]}^{\lambda} \leq a_{(m)[(n)]}^{\lambda}$. This has been proven to hold for $n \leq 5$ \cite{Thrall} \cite{DentSiemons} \cite{McKay} \cite{CIM}, and when $m$ and $n$ are far enough apart \cite{Brion}. But it still remains open in full generality. 

The question may naturally be extended to pairs of integer partitions $\nu$ and $\mu$ as follows.  Consider the binary relation $\nu \trianglelefteq \mu$ on integer partitions, which holds if and only if $h_{\mu}[h_{\nu}] -h_{\nu}[h_{\mu}]$ is Schur positive, or equivalently $a_{\nu[\mu]}^{\lambda}\leq a_{\mu[\nu]}^{\lambda}$ for all $\lambda$. Clearly ''$\trianglelefteq$'' is reflexive and antisymmetric (if we exclude the column partitions $(1^k)$.). The conjecture\footnote{This was first observed by F.Bergeron, personal communication.} here is that it is also transitive; hence it gives us a partial order on integer partitions (see Figure~\ref{fig1}). Clearly, Foulkes conjecture corresponds to the statement that 
  $$ (n)\trianglelefteq (m),\qquad {\mathrm{iff}}\qquad n\leq m,$$
for one part partitions $(m)$ and $(n)$ associated to positive integers $m$ and $n$.

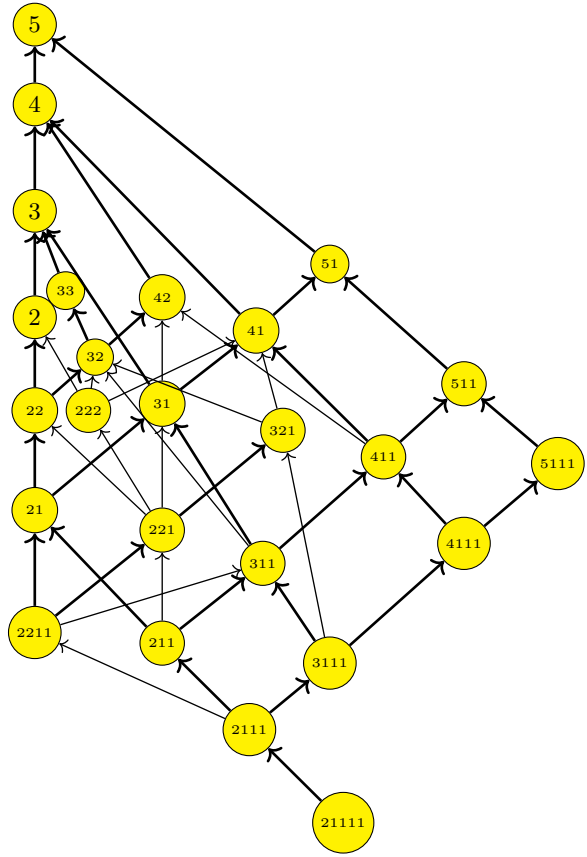
\begin{figure}[!ht]
\begin{center}
\begin{tikzpicture}[line join=bevel,scale=.5]
 \coordinate (P5) at (10bp,790bp);
 \coordinate (P4) at (10bp,730bp);
 \coordinate (P3) at (10bp,650bp);
\coordinate (P2) at (10bp,570bp);
 \coordinate (P33) at (33bp,590bp);
 \coordinate (P42) at (105bp,585bp);
 \coordinate (P51) at (230bp,610bp);
 \coordinate (P511) at (330bp,520bp);
\coordinate (P32) at (55bp,540bp);
 \coordinate (P41) at (175bp,560bp);
 \coordinate (P22) at (10bp,500bp);
 \coordinate (P222)at (50bp,500bp);
 \coordinate (P321) at (195bp,485bp);
 \coordinate (P5111) at (400bp,460bp);
 \coordinate (P411) at (270bp,465bp);
\coordinate (P31)  at (105bp,505bp);
 \coordinate (P21) at (10bp,425bp);
 \coordinate (P221) at (105bp,410bp);
 \coordinate (P4111) at (330bp,400bp);
 \coordinate (P311) at (180bp,385bp);
 \coordinate (P2211) at (10bp,333bp);
 \coordinate (P211) at (105bp,325bp);
 \coordinate (P3111) at (230bp,310bp);
 \coordinate (P2111) at (170bp,260bp);
 \coordinate (P21111) at (240bp,190bp);

\node[circle,fill=yellow,node font=\small] (mu_5) at (P5) [draw]  
	{$5$};
\node[circle,fill=yellow,node font=\small] (mu_4) at (P4) [draw]  
 	{$4$};
\node[circle,fill=yellow,node font=\small] (mu_3) at (P3) [draw]  
	{$3$};
\node[circle,fill=yellow,node font=\tiny,minimum size=.5cm] (mu_51) at (P51) [draw] 
  	{$51$};
\node[circle,fill=yellow,node font=\tiny,minimum size=.5cm] (mu_33) at (P33) [draw]
  	{$33$};
\node[circle,fill=yellow,node font=\tiny,minimum size=.6cm] (mu_42) at (P42) [draw] 
  	{$42$};
\node[circle,fill=yellow,node font=\small] (mu_2) at (P2) [draw] 
	{$2$};
\node[circle,fill=yellow,node font=\tiny] (mu_32) at (P32) [draw] 
  	{$32$};
\node[circle,fill=yellow,node font=\tiny,minimum size=.6cm] (mu_41) at (P41) [draw] 
  	{$41$};
\node[circle,fill=yellow,node font=\tiny] (mu_511) at (P511) [draw] 
  	{$\scriptstyle 511$};
\node[circle,fill=yellow,node font=\tiny] (mu_5111) at (P5111) [draw] 
  	{$\scriptstyle 5111$};
\node[circle,fill=yellow,node font=\tiny] (mu_411) at (P411) [draw] 
  	{$\scriptstyle 411$};
\node[circle,fill=yellow,node font=\tiny] (mu_321) at (P321) [draw]
  	 {$\scriptstyle 321$};
\node[circle,fill=yellow,node font=\tiny,minimum size=.6cm] (mu_31) at (P31) [draw] 
	{$31$};
\node[circle,fill=yellow,node font=\tiny] (mu_222) at (P222) [draw] 
  	{$\scriptscriptstyle 222$};
\node[circle,fill=yellow,node font=\tiny,minimum size=.6cm] (mu_22) at (P22) [draw] 
	{$22$};
\node[circle,fill=yellow,node font=\tiny,minimum size=.6cm] (mu_21) at (P21) [draw]  
	{$21$};
\node[circle,fill=yellow,node font=\tiny] (mu_221) at (P221) [draw] 
  	{$\scriptstyle 221$};
\node[circle,fill=yellow,node font=\tiny] (mu_311) at (P311) [draw] 
  	{$\scriptstyle 311$};
\node[circle,fill=yellow,node font=\tiny] (mu_4111) at (P4111) [draw] 
	 {$\scriptstyle 4111$};
\node[circle,fill=yellow,node font=\tiny] (mu_2211) at (P2211) [draw] 
  	{$\scriptstyle 2211$};
\node[circle,fill=yellow,node font=\tiny] (mu_211) at (P211) [draw] 
 	{$211$};
\node[circle,fill=yellow,node font=\tiny] (mu_3111) at (P3111) [draw] 
  	{$\scriptstyle 3111$};
 \node[circle,fill=yellow,node font=\tiny] (mu_2111) at (P2111) [draw] 
  	{$\scriptstyle 2111$};
 \node[circle,fill=yellow,node font=\tiny] (mu_21111) at (P21111) [draw] 
   	{$\scriptstyle 21111$};
\draw [line width=1pt,black,->] (mu_51) -- (mu_5);		
	\draw [line width=1pt,black,->] (mu_4) -- (mu_5);
\draw [line width=1pt,black,->] (mu_42) -- (mu_4);		
	\draw [line width=1pt,black,->] (mu_41) -- (mu_4);
	\draw [line width=1pt,black,->] (mu_3) -- (mu_4);
\draw [line width=1pt,black,->] (mu_33) -- (mu_3);		
	\draw [line width=1pt,black,->] (mu_31) -- (mu_3);
	\draw [line width=1pt,black,->] (mu_2) -- (mu_3);
\draw [line width=1pt,black,->] (mu_32) -- (mu_33);		
\draw [line width=1pt,black,->] (mu_32) -- (mu_42);		
	\draw [line width=.5pt,black,->] (mu_31) -- (mu_42);
	\draw [line width=.5pt,black,->] (mu_411) -- (mu_42);
\draw [line width=1pt,black,->] (mu_41) -- (mu_51);		
	\draw [line width=1pt,black,->] (mu_511) -- (mu_51);
\draw [line width=.5pt,black,->] (mu_222) -- (mu_2);		
	\draw [line width=1pt,black,->] (mu_22) -- (mu_2);
\draw [line width=1pt,black,->] (mu_411) -- (mu_511);	
	\draw [line width=1pt,black,->] (mu_5111) -- (mu_511);	
\draw [line width=1pt,black,->] (mu_31) -- (mu_41);		
	\draw [line width=.5pt,black,->] (mu_321) -- (mu_41);
	\draw [line width=.5pt,black,->] (mu_222) -- (mu_41);
	\draw [line width=1pt,black,->] (mu_411) -- (mu_41);
\draw [line width=.5pt,black,->] (mu_321) -- (mu_32);		
	\draw [line width=1pt,black,->] (mu_22) -- (mu_32);
	\draw [line width=.5pt,black,->] (mu_222) -- (mu_32);
	\draw [line width=.5pt,black,->] (mu_311) -- (mu_32);
\draw [line width=.5pt,black,->] (mu_221) -- (mu_22);		
	\draw [line width=1pt,black,->] (mu_21) -- (mu_22);
\draw [line width=.5pt,black,->] (mu_221) -- (mu_222);	
\draw [line width=.5pt,black,->] (mu_3111) -- (mu_321);	
	\draw [line width=1pt,black,->] (mu_221) -- (mu_321);
\draw [line width=1pt,black,->] (mu_4111) -- (mu_5111);	
\draw [line width=1pt,black,->] (mu_311) -- (mu_411);	
	\draw [line width=1pt,black,->] (mu_4111) -- (mu_411);	
\draw [line width=1pt,black,->] (mu_311) -- (mu_31);		
	\draw [line width=1pt,black,->] (mu_21) -- (mu_31);
	\draw [line width=.5pt,black,->] (mu_221) -- (mu_31);
\draw [line width=1pt,black,->] (mu_3111) -- (mu_4111);	
\draw [line width=1pt,black,->] (mu_211) -- (mu_21);		
	\draw [line width=1pt,black,->] (mu_2211) -- (mu_21);
\draw [line width=1pt,black,->] (mu_2211) -- (mu_221);	
	\draw [line width=.5pt,black,->] (mu_211) -- (mu_221);
\draw [line width=1pt,black,->] (mu_211) -- (mu_311);	
	\draw [line width=1pt,black,->] (mu_3111) -- (mu_311);
	\draw [line width=.5pt,black,->] (mu_2211) -- (mu_311);
\draw [line width=.5pt,black,->] (mu_2111) -- (mu_2211);	
\draw [line width=1pt,black,->] (mu_2111) -- (mu_211);	
\draw [line width=1pt,black,->] (mu_2111) -- (mu_3111);	
\draw [line width=1pt,black,->] (mu_21111) -- (mu_2111);	

\end{tikzpicture}
 \end{center}\caption{Poset structure on partitions}\label{fig1}
 \end{figure}

	Our aim is to establish some basic results for the relation''$\trianglelefteq$'' using an approach to plethysm via the representation theory of the symmetric group. Indeed, coefficients of Schur expansion of plethysm of the form $h_{\nu}[h_{\mu}]$ occur as multiplicities of irreducible representations in certain modules, denoted here by $M^{\mu[\nu]}$. We construct them using similar techniques as \cite{PagetWildon}, although their construction produce modules associated to plethysm of Schur functions. This technique is used to show that "$\trianglelefteq$" is antisymmetric.

	The original Foulkes’ conjecture can then be interpreted as the existence of some injective homomorphism from $M^{(m)[(n)]}$ to $M^{(n)[(m)]}$ when $m \leq n$. One potential candidate for such a homomorphism is the Foulkes-Howe map \cite{Howe}. While it is not always injective (as shown in \cite{MN}), McKay showed \cite{McKay} that it is sufficient to show that it is injective when $m = n+1$.  Naturally, one is led to consider similar injective homomorphisms from $M^{(\nu)[()]}$ to $M^{(n)[(m)]}$. Hence we construct a generalization of the Foulkes-Howe map, and use it with other techniques to show that our conjecture holds in many cases. In particular, we prove that for all natural numbers $k$ and all partitions $\mu$, we have $(1^k) \trianglelefteq \mu$ and $\mu^k \trianglelefteq \mu$, meaning that this relation contains infinite transitive chains. We also show that the Foulkes’ conjecture implies that  $\mu\trianglelefteq (k)$ for all $k \geq n$, where $n$ is the greatest part of $\mu$. In terms of these modules, we further show some stability properties, like the fact that the irreducible representation associated to $\lambda+(\tilde{m}n)$ appears at least as many times in $M^{(\mu+\tilde{\mu})[\nu]}$ as the one associated to $\lambda$ does in $M^{\mu[\nu]}$.

\section*{Acknowledgments}

The author would like to thank François Bergeron for being the one who first remarked this potential order relation, for his general support and for his many comments and suggestions for this article. 

\section{Main results}

 For symmetric functions, notations used here are mostly those of \cite{Macdonald}. Other classical combinatorial and representation theory notions are recalled in many texts, for example in \cite{Bergeron}.

We describe in section \ref{construction} the $\s_{nm}$-modules $M^{\mu[\nu]}$, whose Frobenius transform of its character is the plethysm $h_{\nu}[h_{\mu}]$, using techniques of \cite{PagetWildon}. The fact that $\nu \trianglelefteq \mu$ is then the existence of an injective morphism $M^{\nu[\mu]} \hookrightarrow M^{\mu[\nu]}$. In section \ref{morphisms}, we show the following result:

\begin{prop}\label{injmorphism}
Let $\nu_1, \nu_2$ and $\mu$ be partitions. If $\nu_1 \trianglelefteq \mu$ and $\nu_2 \trianglelefteq \mu$, then $\nu_1 \sqcup \nu_2 \trianglelefteq \mu$. 
\end{prop}

In section 4, we construct the generalized Foulkes-Howe map $\mathcal{F}_{\mu,\nu}: M^{\mu[\nu]} \rightarrow M^{\nu[\mu]}$. Through sections 3 and 4, we prove the following results:

\begin{thm}\label{bigthm}
\begin{enumerate}[label=(\alph*)]
\item For any positive integer $n$ and partition $\mu$, $\mu^n \trianglelefteq \mu$.
\item If the Foulkes’ conjecture is true up to $n-1$ (\textrm{i.e.} $(n-1) \trianglelefteq (m)$ for all $m \geq n-1$), then for any partition $\mu$ such that $\mu_1 \leq n$, $\mu \trianglelefteq (n)$.
\item For any positive integer $n$ and partition $\mu$, $(1^n) \trianglelefteq \mu$. 
\item For any positive integer $n$ and partition $\mu$, $\mu \sqcup (1^n) \trianglelefteq \mu$. 
\end{enumerate}
\end{thm}

These results give us some structure on $''\trianglelefteq''$. In effect, part 1 and 4 mean that it is an order on the sets $\{ \mu \sqcup (1^k) \ | \ k \in \N\}$ and $\{ \mu^{a^k} \ | \ k \in \N\}$ for any partition $\mu$ and integer $a$. Part 2 mean gives a maximal element for the subset of partitions with parts bounded by an integer $n$. Finally, and most importantly, part 3 tells us that we need to either exclude colum partitions or let them form a join minimum.

In section \ref{ssh}, we describe semistandard homomorphisms as in \cite{James}, and we show how to use them to study the modules $M^{\nu[\mu]}$. Using these morphisms, we prove the following stability properties in section \ref{results2}:

\begin{thm}\label{thm1}
Let $\nu, \mu, \tilde{\mu}$ and $\lambda$ be partitions, where $|\nu| = n$, $|\mu|=m$, $|\lambda|=nm$ and $|\tilde{\mu}|=\tilde{m}$. If $a_{\nu[\mu]}^{\lambda}=r$, then $a_{\nu[\mu+\tilde{\mu}]}^{\lambda+(n\tilde{m})} \geq r$. 
\end{thm}

\begin{thm}\label{thm2}
Let $\nu, \mu$ and $\lambda$ be partitions, where $|\nu| = n$, $|\mu|=m$, $|\lambda|=nm$ and $\ell(\mu)=\widetilde{m}$. If $a_{\nu[\mu]}^{\lambda}=r$, then $a^{\lambda+(2^{n\widetilde{m}})}_{\nu[\mu+(2^{\widetilde{m}})]} \geq r$.
\end{thm}

We conclude with some perspectives for future research about this order.

\section{Description of the modules}\label{construction}

\subsection{Representation theory of the symmetric group}

As a means of deriving properties of plethysms, we consider representations of the symmetric group $\s_n$. The necessary link between the two subjects of study is established via the classical Frobenius transform of characters. This is an isometry which encodes a character as a symmetric function, with the property that irreducible characters are sent to the basis of Schur symmetric function. Using this map, questions about plethysm can be answered using the representation theory of the symmetric group. For an introduction to symmetric functions, see \cite{Macdonald}. We recall here some classical constructions and results about the representation theory of the symmetric group, which may be found for example in \cite{Sagan} or \cite{James}. For a polynomial approach, see \cite{Bergeron}. 

A \textit{partition} is a weakly decreasing sequence of positive integers $\lambda = (\lambda_1, \lambda_2, \ldots)$. If $\sum_i \lambda_i = n$, we say that it is a partition of $n$, denoted $\lambda \vdash n$ or $|\lambda|=n$. The \textit{length} of $\lambda$, denoted $\ell(\lambda)$, is the number of parts of $\lambda$. The \textit{Ferrers diagram} of $\lambda$ is the set $\{ (c,r) \in \mathbb{N} \times \mathbb{N} \ | \ 1 \leq r \leq \ell(\lambda), \ 1 \leq c \leq \lambda_r\}$. We denote $\lambda$ both the partition and the Ferrers diagram. An element $(c,r) \in \lambda$ is called a \textit{cell} of $\lambda$, and is geometrically represented in $\mathbb{N} \times \mathbb{N}$ by the \textit{box} having vertices $(i,j)$, $(i-1,j)$, $(i,j-1)$ and $(i-1,j-1)$ (as in French notation, but with a different numbering of the cells). We say that $(c,r)$ is in the $r^{th}$ \textit{row} and the $c^{th}$ column of $\lambda$. We order the cells with the reverse lexicograpical order $<_{\text{rlex}}$, meaning that $(c,r) <_{\text{rlex}} (c',r')$ if $r < r'$ or $r=r'$ and $c < c'$.

Let $\Omega$ be a set. A \textit{tableau} $\tau$ of \textit{shape} $\lambda$ with entries in $\Omega$ is a map $\tau: \lambda \rightarrow \Omega$. It is often displayed as a filing of the cells of $\lambda$. The $\mathbb{C}$-vector space formally spanned by these tableaux is a $\s_{\Omega}$-module, where $(\rho \cdot \tau)(c,r) = \rho(\tau(c,r))$ for $\rho \in \s_{\Omega}$. If, for $a \in \Omega$, we denote $\gamma_a = |\tau^{-1}(a)|$,and the sequence $\gamma = (\gamma_a)_{a \in \Omega}$ is called the \textit{content} of $\tau$.

If $|\Omega| = |\lambda|=n$, we may consider the set of bijective tableaux $t: \lambda \rightarrow \Omega$, usually denoted by lowercase Latin letters.  Unless specified, we choose $\Omega$ to be the set $\{1,...,n\}$, so that $\s_{\Omega} = \s_n$, but we can make naturally isomorphic constructions for any set of size $n$.   Let $R_{t}$ (resp. $C_{t}$) be the subgroup of $\s_{n}$ which only permutes elements lying in a same row (resp. column) of $t$. Consider the equivalence class $\sim_R$ on bijective tableaux of shape $\lambda$, where $t_1 \sim_R t_2$ if there exists $\sigma \in R_{t_1}$ such that $\sigma \cdot t_1 = t_2$. In that case, $R_{t_1}=R_{t_2}$. A \textit{tabloid} $\{t\}$ is the equivalence class of $t$ under $\sim_R$. Visually, we represent a tabloid by suppressing vertical lines of the diagram. The $\C$-span of these tabloids is called the \textit{permutation module} associated to $\lambda$ and denoted $M^{\lambda}$. In this $\s_n$-module, we define the elements $e(t) = \displaystyle \sum_{\pi \in C_{\lambda}} \sgn(\pi) \{t\}$, called \textit{polytabloids}. It can be shown that for all $\sigma \in \s_n$, $\sigma \cdot e(t) = e(\sigma \cdot t)$. So, the $\mathbb{C}$-span of the polytabloids of shape $\lambda$ is a $\s_n$-module, called the \textit{Specht module} associated to $\lambda$ and denoted $S^{\lambda}$.  We recall some well-known facts about these modules (see Sagan, \cite{Sagan}):

\begin{prop}
\begin{itemize}
\item The set $\{S^{\lambda} \ | \ \lambda \vdash n\}$ is a complete list of non-isomorphic irreducible representations of $\s_n$; 
\item The Frobenius transform of the character of $M^{\lambda}$ is the complete homogeneous symmetric function $h_{\lambda}$;
\item The Frobenius transform of the character of $S^{\lambda}$ is the Schur function $s_{\lambda}$.
\end{itemize}
\end{prop}   

\subsection{Wreath product} \label{Wreath product}

For $n,m \in \mathbb{N}$, let $G_1,G_2,\ldots,G_n$ be subgroups of $\s_m$ and $H$ be a subgroup of $\s_n$. Also, let $f: H \rightarrow \text{Aut}(G_1 \times G_2 \times \ldots \times G_n)$ be the morphism such that for every $h \in H$, $f(h)$ sends $(g_1, ..., g_n) \in G_1 \times G_2 \times \ldots \times G_n$ to $(g_{h^{-1}(1)}, ..., g_{h^{-1}(n)})$. This defines a semidirect product $(G_1 \times \ldots \times G_n) \rtimes_f H$. When $G_1 = G_2 = \ldots = G_n = G$, it is called \textit{wreath product} of $G$ with $H$, denoted $G \wr H$.

Let $V$ be a $G$-module and $W$ be a $H$-module. The tensor product $V^{\otimes n}$ is naturally a $G^n$-module, and it is also a $H$-module, where $H$ acts by permuting the components of a tensor. The interaction of the two actions correspond to the wreath product, so it is in fact a $G\wr H$-module. Also, there is a canonical surjection $G \wr H \twoheadrightarrow H$, so we can construct the inflated $G \wr H$-module $\infl_{H}^{G \wr H} W$. We then define $V \oslash W$ to be the $G \wr H$-module given by the diagonal action on $V^{\otimes n} \otimes \infl_{H}^{G \wr H} W$. When $G= \s_m$ and $H = \s_n$, the operation $\oslash$ mimics the plethysm of symmetric functions, in the following sense \cite{PagetWildon} :

\begin{prop}
Let $V$ be a $\s_m$-module such that the Frobenius transform of its character is $f$, and let $W$ be a $\s_n$-module such that the Frobenius transform of its character is $g$. Then, the Frobenius transform of the character of $(V \oslash W)\big\uparrow_{\s_m \wr \s_n}^{\s_{nm}}$ is the plethysm $g[f]$.
\end{prop}

So, if we want to study the plethysms $h_{\nu}[h_{\mu}]$, we can do it by studying the modules $(M^{\mu} \oslash M^{\nu­})\big\uparrow_{\s_m \wr \s_n}^{\s_{nm}}$.

\subsection{The modules \texorpdfstring{$M^{\mu[\nu]}$}{Mmn}}

We now describe a combinatorial description of the modules $(M^{\mu} \oslash M^{\nu­})\big\uparrow_{\s_m \wr \s_n}^{\s_{nm}}$ for $\mu \vdash m$ and $\nu \vdash n$, described in \cite{PagetWildon} as an intermediary step in the description of their modules. Let $\mathcal{M}_{\mu}$ be the set of injective tableaux $t: [\mu] \rightarrow \{1,2,\ldots,nm\}$. Consider injective tableaux $\bm{T} : [\nu] \rightarrow \mathcal{M}_{\mu}$ such that the union of the entries of each entry $t$ of $\bm{T}$ is $\{1,2,\ldots,nm\}$. We say that such a tableau is a \textit{plethystic tableau} of shape $\nu[\mu]$. Thus, a plethystic tableau of shape $\nu[\mu]$ is a tableau of shape $\nu$ (called the \textit{outer tableau} of $\bm{T}$) whose entries are tableaux of shape $\mu$ (called the \textit{inner tableaux} of $\bm{T}$) and such that each number from 1 to $nm$ appears in exactly one tableau of shape $\mu$. By permuting all the entries at once, the $\mathbb{C}$-span of the plethystic tableaux of shape $\nu[\mu]$ is a $\s_{nm}$-module.

An example of a plethystic tabloid of shape $(2,1)[(2,2)]$ is
\[
\bm{T}=
\ytableausetup{notabloids, aligntableaux=center, boxsize=4em}
\begin{ytableau}
\tydu{1}{5}{4} & \tydu{2}{6}{9} \\
\tydu{3}{10}{12} & \tydu{9}{11}{7}
\end{ytableau} .
\]
Let $t_i$ be the $i^{th}$ inner tableau of $\bm{T}$ (for the order $<_{\text{rlex}}$), and define the group $\bm{R}_{\bm{T}} := (R_{t_1} \times R_{t_2} \times \ldots \times R_{t_n}) \rtimes_f R_{\bm{T}}$ (as in section \ref{Wreath product}). We can define an equivalence relation $\sim_{\bm{R}}$ on plethystic tableaux of shape $\nu[\mu]$, where $\bm{T}_1 \sim_{\bm{R}} \bm{T}_2$ if there exists $\sigma \in \bm{R}_{\bm{T}_1}$ such that $\sigma \cdot \bm{T}_1 = \bm{T}_2$. A \textit{plethystic tabloid} $\{\bm{T}\}$ is the equivalence class of $\bm{T}$ under $\sim_{\bm{R}}$. Visually, we represent it by supressing vertical lines of the outer tableau  and the inner tableaux. We can think of a plethystic tabloid as a "tabloid of tabloids". We also write $\bm{T} = (\{t_1\},\ldots,\{t_n\})_{\nu}$. For example, if $\bm{T}$ is defined as above, then
\[
\{\bm{T}\} =
\ytableausetup{tabloids, aligntableaux=center, boxsize=4em}
\begin{ytableau}
\tbddu{1}{5}{4} & \tbddu{2}{6}{9} \\
\tbddu{3}{10}{12} & \tbddu{9}{11}{7}
\end{ytableau} = \left( \ \tbddu{1}{5}{4} \ , \ \tbddu{2}{6}{9} \ , \ \tbddu{3}{10}{12} \ , \ \tbddu{9}{11}{7} \ \right)_{(2,2)}
\]
This plethystic tabloid is invariant under the permutation $(1,5)(9,11)$, because it permutes pairs of numbers that are in the same row of an inner tableau. It is also invariant under the permutation $(3,9)(10,11)(7,12)$, because this permutation exchanges completely two inner tableaux that are in the same row of the outer tableau. 

Let $M^{\nu[\mu]}$ be the $\s_{nm}$-module generated by the plethysitic tabloids of shape $\nu[\mu]$. 

We have the following result, which is implicit in \cite{PagetWildon}:

\begin{prop}
The representation $M^{\nu[\mu]}$ is isomorphic to the representation $(M^{\mu} \oslash M^{\nu­})\big\uparrow_{\s_m \wr \s_n}^{\s_{nm}}$.
\end{prop}

This means that the Frobenius transform of the character of $M^{\nu[\mu]}$ is $h_{\nu}[h_{\mu}]$. Moreover, the coefficient $a_{\nu[\mu]}^{\lambda}$ of $s_{\lambda}$ in the Schur expansion of $h_{\nu}[h_{\mu}]$ corresponds to the number of submodules of $M^{\nu[\mu]}$ isomorphic to $S^{\lambda}$. By Schur's lemma, we can conclude that $\nu \trianglelefteq \mu$ if and only if there exists an injective homomorphism from $M^{\nu[\mu]}$ to $M^{\mu[\nu]}$.

\section{Results using the tensor product}\label{morphisms}

For two partitions $\nu_1,\nu_2$, we define their union $\nu_1 \sqcup \nu_2$ to be the partition that has the parts of $\nu_1$ followed by the parts of $\nu_2$, and then reordered to obtain a partition. For example, $(2,2) \sqcup (3,1) = (3,2,2,1)$. Also, if $t_1$ is a tableau of shape $\nu_1$, $t_2$ is a tableau of shape $\nu_2$, we define their union $t_1 \sqcup t_2$ to be the tableau obtained by stacking $t_1$ and $t_2$, and then reordering the rows so that it is a tableau of shape $\nu_1 \sqcup \nu_2$. We can also make the same construction with tabloids.
For example, we have
\[
\{t_1\} = 
\ytableausetup{tabloids}
\begin{ytableau}
2 & 3 \\
1 & 8
\end{ytableau}; \qquad \{t_2\} = 
\begin{ytableau}
7 \\
5 & 4 & 6
\end{ytableau}
\qquad \implies \qquad
\{t_1\} \sqcup \{t_2\} = 
\begin{ytableau}
7 \\
2 & 3\\
1 & 8\\
5 & 4 & 6
\end{ytableau}.
\]
It is easy to see that $\{t_1\} \sqcup \{t_2\} = \{t_1 \sqcup t_2\}$, as putting rows on top of each other instead of next to each other doesn't change the row symmetries. Moreover, if we know $\nu_1$ and $\nu_2$, this construction is invertible. This gives the following result:

\begin{prop} \label{Isotens}
Let $\mu$ be a partition of $m$ and $\nu$ be a partition of $n$. Suppose that there is $\nu_1,\nu_2$ such that $\nu = \nu_1 \sqcup \nu_2$, with $\nu_1 \vdash n_1$ and $\nu_2 \vdash n_2$. Then, $M^{\nu[\mu]} \cong \left( M^{\nu_1[\mu]} \otimes M^{\nu_2[\mu]}\right) \big\uparrow_{\s_{n_1m} \times \s_{n_2m}}^{\s_{nm}}$.
\begin{proof}
An element of $\left( M^{\nu_1[\mu]} \otimes M^{\nu_2[\mu]}\right) \big\uparrow_{\s_{mn_1} \times \s_{mn_2}}^{\s_{mn}}$ is a tensor $\{\bm{T}_1\} \otimes \{\bm{T}_2\}$, where $\{\bm{T}_1\}$ is a plethystic tabloid of shape $\nu_1[\mu]$, $\{\bm{T}_2\}$ is a plethystic tabloid of shape $\nu_2[\mu]$, both with entries in $\{1,\ldots,nm\}$, and such that each number from $1$ to $nm$ appears in exactly one of the two plethystic tabloids. Knowing this, it is easy to see that the map $\{\bm{T}_1\} \otimes \{\bm{T}_2\} \mapsto \{\bm{T}_1 \sqcup \bm{T}_2\}$ is a bijection. 
\end{proof}
\end{prop}

For example, if $\mu=(3,2)$, $\nu=(2,2)$, and $\nu_1 = \nu_2 = (2)$, then
\[
\{\bm{T}_1\} \otimes \{\bm{T}_2\} =
\ytableausetup{boxsize=4.5em}
\begin{ytableau}
\tytd{4}{19}{11}{8}{14} & \tytd{13}{5}{9}{7}{20}
\end{ytableau} \otimes
\ytableausetup{boxsize=4.5em}
\begin{ytableau}
\tytd{3}{17}{18}{1}{2} & \tytd{6}{12}{10}{15}{16}
\end{ytableau}
\quad \mapsto \quad 
\{\bm{T}_1 \sqcup \bm{T}_2\} = 
\ytableausetup{boxsize=4.5em}
\begin{ytableau}
\tytd{3}{17}{18}{1}{2} & \tytd{6}{12}{10}{15}{16} \\
\tytd{4}{19}{11}{8}{14} & \tytd{13}{5}{9}{7}{20}
\end{ytableau}.
\]

When we decompose $\mu$ in the same way, we have a weaker result:

\begin{prop}
Let $\mu$ be a partition of $m$ and $\nu$ be a partition of $n$. Suppose that there is $\mu_1,\mu_2$ such that $\mu = \mu_1 \sqcup \mu_2$, with $\mu_1 \vdash m_1$ and $\mu_2 \vdash m_2$. Then, we have an injection $\left(M^{\nu[\mu_1]} \otimes M^{\nu[\mu_2]} \right) \big\uparrow_{\s_{nm_1} \times \s_{nm_2}}^{\s_{nm}} \hookrightarrow M^{\nu[\mu]}$.
\begin{proof}
An element of $\left(M^{\nu[\mu_1]} \otimes M^{\nu[\mu_2]} \right) \big\uparrow_{\s_{nm_1} \times \s_{nm_2}}^{\s_{nm}}$ is of the form $\{\bm{T}_1\} \otimes \{\bm{T}_2\}$, where $\bm{T}_1, \bm{T}_2$ are plethystic tabloids of shape $\nu[\mu_1]$ and $\nu[\mu_2]$, respectively, both with entries in $\{1,\ldots,nm\}$ and such that each number from $1$ to $nm$ appears in exactly one of the two plethystic tabloids. Choose representatives $\bm{T}_1 = (\{t_1\}, \ldots, \{t_n\})$ and $\bm{T}_2 = (\{t_1'\}, \ldots, \{t_n'\})$ (so we choose an order for the rows of the outer tabloid), and let $\bm{T} = (\{t_1 \sqcup t_1'\}, \ldots, \{t_k \sqcup t_k'\})$. Every choice of representative gives a tableau of the form $\sigma \cdot \bm{T}$ for $\sigma \in R_{\bm{T}_1} \times R_{\bm{T}_2}$, where this group acts by permuting inner tabloids in the same row. So, the map $\{\bm{T}_1\} \otimes \{\bm{T}_2\} \mapsto \displaystyle \sum_{\sigma \in R_{\bm{T}_1} \times R_{\bm{T}_2}} \sigma \cdot \bm{T}$ is well-defined, and is injective.
\end{proof}
\end{prop}

For example, if $\mu=(3,2)$ and $\nu=(2,2)$ as before, then $\mu = (3) \sqcup (2)$, then
\[
\ytableausetup{boxsize=4.5em}
\begin{ytableau}
\tyt{3}{17}{18} & \tyt{6}{10}{12} \\
\tyt{4}{11}{19} & \tyt{5}{9}{13}
\end{ytableau} \ \otimes \ 
\ytableausetup{boxsize=3.5em}
\begin{ytableau}
\tyd{1}{2} & \tyd{15}{16} \\
\tyd{8}{14} & \tyd{7}{20}
\end{ytableau}
\quad \mapsto \quad \displaystyle \sum_{\sigma \in R_{\bm{T}_1} \times R_{\bm{T}_2}} \sigma \cdot
\ytableausetup{boxsize=4.5em}
\begin{ytableau}
\tytd{3}{17}{18}{1}{2} & \tytd{6}{10}{12}{15}{16} \\
\tytd{4}{11}{19}{8}{14} & \tytd{5}{9}{13}{7}{20}
\end{ytableau} 
\]

By combining these two results, we obtain a proof of proposition \ref{injmorphism}:

\begin{proof}[Proof of proposition \ref{injmorphism}]
For $\nu = \nu_1 \sqcup \nu_2$, the fact that $\nu_1 \trianglelefteq \mu$ and $\nu_2 \trianglelefteq \mu$ mean that we have injections $\mathcal{F}_1: M^{\nu_1[\mu]} \hookrightarrow M^{\mu[\nu_1]}$ and $\mathcal{F}_2: M^{\nu_2[\mu]} \hookrightarrow M^{\mu[\nu_2]}$. Because induction and tensor product are functorial constructions , we can construct the morphism $\mathcal{F} = \left(\mathcal{F}_1 \otimes \mathcal{F}_2\right) \big\uparrow_{\s_{mn_1} \times \s_{mn_2}}^{\s_{mn}}$. Using previous propositions, we have the composition:
\[
M^{\mu[\nu_1 \sqcup \nu_2]} \tilde{\rightarrow} \left( M^{\mu[\nu_1]} \otimes M^{\mu[\nu_2]}\right) \big\uparrow_{\s_{mn_1} \times \s_{mn_2}}^{\s_{mn}} \xrightarrow{\mathcal{F}} \left( M^{\nu_1[\mu]} \otimes M^{\nu_2[\mu]}\right) \big\uparrow_{\s_{mn_1} \times \s_{mn_2}}^{\s_{mn}} \hookrightarrow M^{(\nu_1 \sqcup \nu_2)[\mu]}.
\]
As both functors are exacts (representations are vector spaces), the injectivity of $\mathcal{F}_1$ and $\mathcal{F}_2$ implies the injectivity of $\mathcal{F}$, hence the injectivity of the composition.
\end{proof}

This proposition has useful corollaries, which are part $(a)$ and $(b)$ of theorem \ref{bigthm}:

\begin{cor}
Let $\mu^n = \underbrace{\mu \sqcup \ldots \sqcup \mu}_{n \textrm{ times}}$. Then, there is an injective morphism $M^{\mu^n[\mu]} \hookrightarrow M^{\mu[\mu^n]}$. 
\begin{proof}
When $n=1$, take the identity map. Suppose that it is true for $n$. Then, $\mu^{n+1} =  \mu^n \sqcup \mu$, so by induction and proposition \ref{injmorphism}, there exists an injective homomorphism $M^{\mu^{n+1}[\mu]}  \hookrightarrow M^{\mu[\mu^{n+1}]}$.
\end{proof}
\end{cor}

\begin{cor}
Suppose that Foulkes' conjecture is true up to $n-1$, \textrm{i.e.} there exists an injective homomorphism $M^{(n-1)[(m)]} \hookrightarrow M^{(m)[(n-1)]}$ for all $m \geq n - 1$. Choose $\mu$ such that its largest part is smaller or equal to $n$. Then, there exists an injective homomorphism $M^{\mu[(n)]} \hookrightarrow M^{(n)[\mu]}$.
\begin{proof}
If $\mu = (\mu_1, \ldots, \mu_k)$, then $\mu = (\mu_1) \sqcup \ldots \sqcup (\mu_k)$. If $\mu_i = n$, then the injective morphism $M^{(n)[(n)]} \hookrightarrow M^{(n)[(n)]}$ is the indentity. If $\mu_i < n$, then by the assumption on Foulkes' conjecture, there exists an injective homomorphism $M^{(\mu_i)[(n)]} \hookrightarrow M^{(n)[(\mu_i)]}$. Using proposition \ref{injmorphism} iteratively, we can then construct an injective homomorphism $M^{\mu[(n)]} \hookrightarrow M^{(n)[\mu]}$. 
\end{proof}
\end{cor}

This last corollary gives some structure on $\trianglelefteq$, given that Foulkes' conjecture is true. In effect, it implies that the set $\{\mu \ | \ \text{The first part of } \mu \text{ is } \leq n\}$ has $(n)$ as a maximal element. So, we can study the relation $\trianglelefteq$ on the subsets $\mathcal{P}_n = \{\mu \ | \ \text{The first part of } \mu \text{ is exactly } n\}$, and then study the relations between the different $\mathcal{P}_n$. A first result for $\mathcal{P}_2$, from the author and Maas-Gariépy, can be found in \cite{FlorenceEtMoi}, and the author hopes to generalize this result in the future.

\section{Results using homomorphisms}

\subsection{Projection from permutation modules}\label{permmod}

For a partition $\mu$, let $\mu^n$ be $\underbrace{\mu \sqcup ... \sqcup \mu}_{n \ \text{times}}$, and consider the permutation module $M^{\mu^n}$. By iteravely applying proposition \ref{Isotens} (for $\nu=(1)$), this module is isomorphic to the module $(M^{\mu})^{\otimes n}\big\uparrow_{(\s_{n})^m}^{\s_{nm}}$. So, we can write an element of $M^{\mu^n}$ as $\{t_1\} \otimes ... \otimes \{t_n\}$, where each $\{t_i\}$ is an tabloid of shape $\mu$ with entries in $\{1,\ldots,nm\}$ and each number from $1$ to $nm$ is used exactly once. We define the projection
\[
\begin{array}{ccll}
\phi: & M^{\mu^n} & \xtwoheadrightarrow & M^{\nu[\mu]} \\
         & \{t_1\} \otimes ... \otimes \{t_n\} & \longmapsto           & (\{t_1\},\ldots,\{t_n\})_{\nu}
\end{array}.
\]
Reciprocally, there is an injective homomorphism
\[
\begin{array}{crlc}
\widetilde{\phi}: & M^{\nu[\mu]} & \hooklongrightarrow & M^{\mu^n} \\
                    & \{\bm{T}\} = (\{t_1\},\ldots,\{t_n\})_{\nu}	& \longmapsto	         & \displaystyle \frac{1}{\left|R_{\bm{T}}\right|} \sum_{\sigma \in R_{\bm{T}}} \sigma \cdot \{t_1\} \otimes \ldots \otimes \{t_n\}
\end{array},
\]
We can easily see that $\phi \circ \widetilde{\phi} = \text{Id}$. For example, if $\nu = (2,1)$ and $\mu = (2,2)$, then
\begin{align*}
\phi \left( \widetilde{\phi} \left( \ytableausetup{tabloids, boxsize=4em}\begin{ytableau} \tbddd{4}{12}{3}{5} \\ \tbddd{1}{9}{7}{11} & \tbddd{6}{10}{2}{8} \end{ytableau} \right) \right) &= \phi \left( \frac{1}{2} \left( \tbddd{1}{9}{7}{11} \otimes \tbddd{6}{10}{2}{8} \otimes \tbddd{4}{12}{3}{5} +  \tbddd{6}{10}{2}{8} \otimes\tbddd{1}{9}{7}{11} \otimes \tbddd{4}{12}{3}{5} \right) \right) \\ &= \ytableausetup{tabloids, boxsize=4em} \begin{ytableau} \tbddd{4}{12}{3}{5} \\ \tbddd{1}{9}{7}{11} & \tbddd{6}{10}{2}{8} \end{ytableau}.
\end{align*}

These two morphisms are generalization of similar constructions in \cite{deBoeck} in the case $\mu=(m)$. They are used in both section \ref{GFHM} and section \ref{ssh}.

\subsection{Generalized Foulkes-Howe map}\label{GFHM}

We would like to construct a candidate for an homomorphism $M_{\mu}^{\nu} \hookrightarrow M_{\nu}^{\mu}$. We describe in the previous section a morphism from $M_{\mu}^{\nu}$ to $M^{\mu^n}$, and another from $M^{\nu^m}$ to $M_{\nu}^{\mu}$. The only missing part is a morphism from $M^{\mu^n}$ to $M^{\nu^m}$.

Let $\{t_1\} \otimes ... \otimes \{t_n\}$ be an element of $M^{\mu^n}$. Denote $x_{i,j}$ the $j^{th}$ entry of $t_i$ (for the order $<_{\text{rlex}}$ and any choice of representative). The set of all $x_{i,j}$ is equal to the set $\{1,\ldots,nm\}$. Then, for $1 \leq j \leq m$, consider the tableau $s_j$ of shape $\nu$ such that its $i^{th}$ entry (for $<_{\text{rlex}}$) is $x_{i,j}$. We define the following morphism:
\[
\begin{array}{ccll}
\Psi: & M^{\mu^n} & \longrightarrow & M^{\nu^m} \\
         & \{t_1\} \otimes ... \otimes \{t_n\} & \longmapsto   & \displaystyle \sum_{\sigma \in R_{t_1} \times ... \times R_{t_n}} \sigma \cdot \left( \{s_1\} \otimes ... \otimes \{s_m\} \right)
\end{array}.
\] 
For example, if $\mu = (2,2)$, $\nu = (2,1)$, and 
\[
\{t_1\} \otimes \{t_2\} \otimes \{t_3\} =
\ytableausetup{tabloids}
\begin{ytableau}
\textcolor{red}{3} & \textcolor{red}{4} \\
\textcolor{violet}{1} & \textcolor{violet}{2}
\end{ytableau} \ \otimes \
\begin{ytableau}
\textcolor{teal}{7} & \textcolor{teal}{8} \\
\textcolor{orange}{5} & \textcolor{orange}{6}
\end{ytableau}
\  \otimes \
\begin{ytableau}
\textcolor{blue}{11} & \textcolor{blue}{12} \\
\textcolor{green}{9} & \textcolor{green}{10}
\end{ytableau},
\]
then its image by $\Psi$ is
\[
\displaystyle \sum_{\sigma \in R_{t_1} \times R_{t_2} \times R_{t_3}} \sigma \cdot \left( \{s_1\} \otimes \{s_2\} \otimes \{s_3\} \otimes \{s_4\} \right) = 
\displaystyle \sum_{\sigma \in R_{t_1} \times R_{t_2} \times R_{t_3}} \sigma \cdot \left(
\begin{ytableau}
\textcolor{green}{9}  \\
\textcolor{violet}{1} & \textcolor{orange}{5}
\end{ytableau} 
\ \otimes \
\begin{ytableau}
\textcolor{green}{10}  \\
\textcolor{violet}{2} & \textcolor{orange}{6}
\end{ytableau}
\ \otimes \
\begin{ytableau}
\textcolor{blue}{11}  \\
\textcolor{red}{3} & \textcolor{teal}{7}
\end{ytableau}
\ \otimes \
\begin{ytableau}
\textcolor{blue}{12}  \\
\textcolor{red}{4} & \textcolor{teal}{8}
\end{ytableau} \right)
\]
Now consider the composition of homomorphisms $\phi \circ \Psi \circ \widetilde{\phi}: M^{\nu[\mu]} \rightarrow M^{\mu[\nu]}$. When $\mu=(m)$, $\nu=(n)$, this map is the \textit{Foulkes-Howe map} stated in the language of symmetric groups. This map was developed as a tool to prove the Foulkes’ conjecture. For this reason, we call this homomorphism the \textit{generalized Foulkes-Howe map}, and we denote it $\mathcal{F}_{\nu,\mu}$. We can have a better understanding of this composition by using the following diagram:

\begin{center}
\begin{tikzcd}
M^{\nu[\mu]}
	\arrow[r, dashed, "\mathcal{F}_{\nu,\mu}"] 
	\arrow[d, "\widetilde{\phi}"]
& M^{\mu[\nu]} \\
M^{\mu^n}
	\arrow[r, "\Psi"]
& M^{\nu^m}
	\arrow[u, "\phi"] \\
\end{tikzcd}
\end{center}

We have the following result for the (regular) Foulkes-Howe map, proven by McKay \cite{McKay}:

\begin{thm}[Propagation theorem]
If $\mathcal{F}_{(n),(m)}$ is injective, then the map $\mathcal{F}_{(n),(m+1)}$ is also injective.
\end{thm}

Moreover, we know that for $n \leq 5$, the map $\mathcal{F}_{(n),(n+1)}$ is injective \cite{CIM}. Using the theorem inductively, $\mathcal{F}_{(n), (m)}$ is injective for $m\geq n+1$. The case $m=n$ is trivial (although the map $\mathcal{F}_{(n),(n)}$ is not always injective, see \cite{CIM}), so the Foulkes’ conjecture holds for $n \leq 5$. In the future, we hope to find propagation theorems for other (conjectured) chains in the relation $\trianglelefteq$, because in can help to show that this relation is transitive. For example, the author believes that for any partition $\mu$, the subset $\{\mu + (k) \ | \ k \in \N\}$ is a transitive chain of the relation, with $\mu + (k) \trianglelefteq \mu + (k+1)$. When $\mu$ is the empty partition, this corresponds to the Foulkes' conjecture, so we think we can generalize the propagation theorem to this case, and probably more.

\subsection{First results}\label{results1}

We can use the generalized Foulkes-Howe map and theorem \ref{injmorphism} to prove the two last parts of \ref{bigthm}. Part $(c)$ follows from the following proposition:

\begin{prop}
For any positive integer $n$ and partition $\mu$, there is an injection $M^{(1^n)[\mu]} \hookrightarrow M^{\mu[(1^n)]}$. 
\begin{proof}
We use the generalized Foulkes-Howe map, so we have the morphisms
\begin{center}
\begin{tikzcd}
\mathcal{F}_{(1^n),\mu}: M_{(1^n)[\mu]} 
	\arrow[r, hookrightarrow, "\widetilde{\phi}"] 
& M^{\mu^n} 
	\arrow[r, "\Psi"] 
& M^{(1^n)^m} 
	\arrow[r, twoheadrightarrow, "\phi"]
& M^{\mu[(1^n)]} 
\end{tikzcd}
\end{center}
Let $\{\bm{T}\} = (\{t_1\},\ldots,\{t_n\})_{(1^n)}$ be a plethystic tabloid of shape $(1^n)$. Then, every inner tabloid is in its own row, so $R_{\bm{T}}$ is the trivial group. This means that $\widetilde{\phi}$, which maps $\{\bm{T}\}$ to $\{t_1\} \otimes \ldots \otimes \{t_n\}$, is an isomorphism. 

Fix representatives $t_1, \ldots, t_n$, and define tableaux $s_j$ of shape $(1^n)$ such that $s_j(1,i)$ is the $j^{th}$ value of $t_i$ (for $<_{\text{rlex}}$).Then, 
\[
\Psi(\{t_1\} \otimes \ldots \otimes \{t_n\}) = \displaystyle \sum_{\sigma \in R_{t_1} \times \ldots \times R_{t_n}} \sigma \cdot \left( \{s_1\} \otimes \ldots \otimes \{s_m\} \right).
\]
Note that for any $j$, the tabloid $\{s_j\}$ is of shape $(1^n)$, so $s_j$ is the only element of the equivalence class. Finally applying $\phi$ to this result, it follows that 
\[
\mathcal{F}_{(1^n),\mu}(\bm{T}) = \dfrac{1}{|R_{\bm{T}}|} \displaystyle \sum_{\sigma \in R_{t_1} \times \ldots \times R_{t_n}} \sigma \cdot \left( \{s_1\}, \ldots, \{s_m\} \right)_{\mu}.
\]
Let $\bm{S} = \left( \{s_1\}, \ldots, \{s_m\} \right)_{\mu}$. The entries that are in row $r$ of $\{t_i\}$ are mapped to the row $i$ of a certain inner tabloid of $\bm{S}$ situated in the $r^{th}$ row of the outer tabloid. Which one is not important, as the sum is over all possible ways to do that. As there is only one entry in each row of a tabloid of shape $(1^n)$, it means that we can recover $\{\bm{T}\}$ from $\{\bm{S}\}$, and the sum is there so that the map is well-defined. Thus, the map is injective. 
\end{proof}
\end{prop}

For example, if $\mu=(2,2)$ and $n=2$, we have that
\[
\mathcal{F}_{(1^n),\mu} \left( \ \ytableausetup{boxsize=4em, tabloids} \begin{ytableau} \tbdtd{\textcolor{orange}{5}}{\textcolor{orange}{6}}{\textcolor{orange}{8}}{\textcolor{teal}{1}}{\textcolor{teal}{10}} \\ \tbdtd{\textcolor{magenta}{2}}{\textcolor{magenta}{3}}{\textcolor{magenta}{9}}{\textcolor{red}{4}}{\textcolor{red}{7}} \end{ytableau} \ \right) = \displaystyle \frac{1}{3!2!} \sum_{\sigma} \sigma \cdot \left( \ \ytableausetup{boxsize=4em, tabloids} \begin{ytableau} \tbduu{\textcolor{red}{4}}{\textcolor{teal}{1}} & \tbduu{\textcolor{red}{7}}{\textcolor{teal}{10}} \\ \tbduu{\textcolor{magenta}{2}}{\textcolor{orange}{5}} & \tbduu{\textcolor{magenta}{3}}{\textcolor{orange}{6}} & \tbduu{\textcolor{magenta}{9}}{\textcolor{orange}{8}} \end{ytableau} \ \right),
\]
where the sum is over all $\sigma \in \s_{10}$ permuting numbers of the same color. Clearly, only this plethystic tabloid can have this image.

As a corollary, we obtain part $(d)$ of theorem \ref{bigthm}:

\begin{cor}
For any positive integer $n$ and partition $\mu$, there is an injection $M^{(\mu \sqcup 1^n)[\mu]} \hookrightarrow M^{\mu[\mu \sqcup 1^n]}$.
\begin{proof}
Take the indentity morphism $M^{\mu[\mu]} \rightarrow M^{\mu[\mu]}$ and the generalized Foulkes-Howe map $\mathcal{F}_{1^n,\mu}: M^{(1^n)[\mu]} \hookrightarrow M^{\mu[(1^n)]}$. Both are injectives (first is obvious, second follows from previous proposition), so by propostion \ref{injmorphism}, we can construct an injective morphism $ M^{(\mu \sqcup (1^n))[\mu]} \hookrightarrow M^{\mu[\mu \sqcup (1^n)]}$. 
\end{proof}
\end{cor} 


This seems to be the only instance where the image of a plethystic tabloid under the generalized Foulkes-Howe map is unique. In general, we have to decompose the module $M^{\nu[\mu]}$ into irreducible representations (which is a hard problem), and verify that each irreducible is not sent to the trivial module. The next section explains how to find the irreducible representations contained in $M^{\nu[\mu]}$.

\section{Decomposition via semistandard homomorphisms}\label{ssh}

\subsection{Semistandard homomorphisms}

We use here the description of \cite{James} describing the number of irreducible subrepresentations of $M^{\mu}$ isomoprhic to $S^{\lambda}$, and mostly use its notations.

Let $\lambda \vdash n$, and let $\tau$ be a tableau of content $\mu$. We denote $T(\lambda,\mu)$ the set of such tableaux. Remark that up to a relabelling of the entries, we can always choose $\mu$ to be a partition. 

Let $t: \lambda \rightarrow \{1,2,\ldots,n\}$ be a bijective tableau. We define an action of $\s_n$ on $T(\lambda,\mu)$, where $(\sigma \cdot \tau)(c,r) = \tau\left( (\sigma^{-1} \cdot t)(c,r) \right)$. There is an isomorphism $f_t: \C\{T(\lambda,\mu)\} \rightarrow M^{\mu}$ such that $f_t(\tau)$ is the tabloid of shape $\mu$ such that if $\tau(c,r)=i$, then the number $t(c,r)$ is in its $i^{th}$ row. 

For example, if $\lambda=(3,2)$ and $\mu=(2,2,1)$  we have
\[
\ytableausetup{notabloids}
t = 
\begin{ytableau}
\textcolor{red}{4}	&\textcolor{red}{5} \\
\textcolor{blue}{1}	&\textcolor{blue}{2}	&\textcolor{orange}{3}
\end{ytableau} \qquad
\tau = 
\begin{ytableau}
\textcolor{red}{2}	& \textcolor{red}{2} \\
\textcolor{blue}{1}	&\textcolor{blue}{1}	& \textcolor{orange}{3}
\end{ytableau}
 \quad \implies \quad
f_{t}(\tau) = \ytableausetup{tabloids} \begin{ytableau} \textcolor{orange}{3} \\ \textcolor{red}{4} & \textcolor{red}{5} \\ \textcolor{blue}{1} & \textcolor{blue}{2} \end{ytableau}.
\]

Let $\sim_R$ be the equivalence relation on $T(\lambda,\mu)$ such that $\tau_1 \sim_R \tau_2$ if there exists $\sigma \in R_{t}$ such that $\sigma \cdot \tau_1 = \tau_2$. This allows us to define the following homomorphism:
\[
\begin{array}{crll}
\widehat{\Theta}_{\tau}: & M^{\lambda} & \longrightarrow & M^{\mu} \\
				&  \{t\}	    & \mapsto		 & \displaystyle \sum_{\tau' \sim_R \tau} f_t(\tau')
\end{array}
\]
For example, using the same $\lambda, \mu, t$ and $\tau$ as in the previous example, we have
\[
\ytableausetup{tabloids} \widehat{\Theta}_{\tau} \left( \begin{ytableau} 4	&5 \\ 1	& 2	& 3 \end{ytableau} \right)  \quad \mapsto \quad
 f_{t} \ \left(
\ytableausetup{notabloids}
\begin{ytableau}
2	& 2 \\
1	& 1	& 3
\end{ytableau}
+
\begin{ytableau}
2	& 2 \\
1 	& 3	& 1
\end{ytableau}
+
\begin{ytableau}
2	& 2 \\
3	& 1	& 1
\end{ytableau}
\ \right) \quad = \quad
\ytableausetup{tabloids}
\begin{ytableau} 3 \\ 4 & 5 \\ 1 & 2 \end{ytableau} + \begin{ytableau} 2 \\ 4 & 5 \\ 1 & 3 \end{ytableau} + \begin{ytableau} 1 \\ 4 & 5 \\ 2 & 3 \end{ytableau}
\]
Denote $\Theta_{\tau}$ for the restriction of $\widehat{\Theta}_{\tau}$ to $S^{\lambda}$, so that $\Theta_{\tau}(e(t)) = \displaystyle \sum_{\substack{\tau' \sim_R \tau \\ \pi \in C_t}} \sgn(\pi) f_t(\pi \cdot \tau')$. Taking the previous example, we have
\begin{align*}
\ytableausetup{notabloids} e \left( \begin{ytableau} 4	&5 \\ 1	& 2	& 3 \end{ytableau} \right) &= f_t\left(
\begin{ytableau}
2	& 2 \\
1	& 1	& 3
\end{ytableau}
+
\begin{ytableau}
2	& 2 \\
1 	& 3	& 1
\end{ytableau}
+
\begin{ytableau}
2	& 2 \\
3	& 1	& 1
\end{ytableau}
+\begin{ytableau}
1	& 1 \\
2	& 2	& 3
\end{ytableau}
+
\begin{ytableau}
1	& 3 \\
2 	& 2	& 1
\end{ytableau}
+
\begin{ytableau}
3	& 1 \\
2	& 2	& 1
\end{ytableau}
\right. \\
& \left. \qquad \quad -
\begin{ytableau}
1	& 2 \\
2	& 1	& 3
\end{ytableau}
-
\begin{ytableau}
1	& 2 \\
2 	& 3	& 1
\end{ytableau}
-
\begin{ytableau}
3	& 2 \\
2	& 1	& 1
\end{ytableau}
- 
\begin{ytableau}
2	& 1 \\
1	& 2	& 3
\end{ytableau}
-
\begin{ytableau}
2	& 3 \\
1 	& 2	& 1
\end{ytableau}
-
\begin{ytableau}
2	& 1 \\
3	& 2	& 1
\end{ytableau}
 \right) \\ \ytableausetup{tabloids}
&= \begin{ytableau} 3 \\ 4 & 5 \\ 1 & 2 \end{ytableau} + \begin{ytableau} 2 \\ 4 & 5 \\ 1 & 3 \end{ytableau} + \begin{ytableau} 1 \\ 4 & 5 \\ 2 & 3 \end{ytableau} + \begin{ytableau} 3 \\ 1 & 2 \\ 4 & 5 \end{ytableau} + \begin{ytableau} 5 \\ 1 & 2 \\ 3 & 4 \end{ytableau} + \begin{ytableau} 4 \\ 1 & 2 \\ 3 & 5 \end{ytableau}  \\
&\quad - \begin{ytableau} 3 \\ 1 & 5 \\ 2 & 4 \end{ytableau} - \begin{ytableau} 2 \\ 1 & 5 \\ 3 & 4 \end{ytableau} - \begin{ytableau} 4 \\ 1 & 5 \\ 2 & 3 \end{ytableau} - \begin{ytableau} 3 \\ 2 & 4 \\ 1 & 5 \end{ytableau} - \begin{ytableau} 5 \\ 2 & 4 \\ 1 & 3 \end{ytableau} - \begin{ytableau} 1 \\ 2 & 4 \\ 3 & 5 \end{ytableau} 
\end{align*}

Note that for a fixed $\tau'$, $\displaystyle \sum_{\pi \in C_t} \sgn(\pi) \pi \cdot f_t(\tau') = 0$ if and only if $\tau'$ has a column containing twice the same entry, because it is the only instance where both an odd an an even permutation of $C_t$ gives the same tableau. 

We say that a tableau $\tau$ with entries in $\mathbb{N}*$ is \textit{semistandard} if it is weakly increasing in the rows (from left to right) and strictly increasing in the columns (from bottom to top). When $\tau \in T(\lambda,\mu)$ is semistandard, we say that $\Theta_{\tau}$ is a \textit{semistandard homomorphism}. These homomorphisms have the nice following property:

\begin{prop}
The set $\{ \Theta_{\tau} \ | \ \tau \in T(\lambda,\mu) \ \text{semistandard} \}$ is a basis for $\HOM_{\s_n}(S^{\lambda},M^{\mu})$.
\end{prop}

By Schur's lemma, it implies that the number of irreducible subrepresentations of $M^{\mu}$ isomorphic to $S^{\lambda}$ correspond to the number of semistandard tableaux of shape $\lambda$ and content $\mu$.

\subsection{Application for the modules \texorpdfstring{$M_{\mu}^{\nu}$}{Mmn}}

We want to use the semistandard homomorphisms to study the modules $M_{\mu}^{\nu}$. First, combining the semistandard homomorphisms and the morphisms of section \ref{permmod}, we obtain the following setting:

\begin{center}
\begin{tikzcd}
S^{\lambda}
	\arrow[r]
	\arrow[rr, bend left, "\Theta_{\tau}"] 
& M^{\nu[\mu]}
	\arrow[r, hook, "\widetilde{\phi}"]
	\arrow[rr, bend right, "\phi \circ \widetilde{\phi} = \text{Id}"]
&M^{\mu^n}
	\arrow[r, two heads, "\phi"]
&M^{\nu[\mu]}
\end{tikzcd}
\end{center}

So to understand the decomposition of $M^{\nu[\mu]}$, let $\lambda$ be a partition of $nm$, and $\tau$ a semistandard tableau of shape $\lambda$ and content $\mu^n$. Define $\overline{\Theta}_{\tau} = \phi \circ \Theta_{\tau}$. The diagram means that the morphisms $\overline{\Theta}_{\tau}$, for $\tau \in T(\lambda,\mu)$, is a generating set of $\HOM_{\s_{nm}}(S^{\lambda}, M^{\nu[\mu]})$. The cardinality of a basis gives the number of copies of $S^{\lambda}$ in $M^{\nu[\mu]}$. However, finding such a basis is a hard problem. 

\section{Stability properties}\label{results2}

We use semistandard homomorphisms to prove that some sequences of plethysm coefficients are increasing. To do so, we generalize the arguments of de Boeck in \cite{deBoeck}. 

For two tableaux $\tau_1,\tau_2$, we define the \textit{join} of the two tableaux, denote $\tau_1 \vee \tau_2$, to be the tableau such that the row $i$, when read from left to right, consists of the row $i$ of $\tau_1$ followed by the row $i$ of $\tau_2$. For example:
\[
\ytableausetup{notabloids} \tau_1 = \begin{ytableau} 3 & 4 \\ 2 & 2 & 3 \\ 1 & 1 & 1 \end{ytableau}; \qquad \tau_2 = \begin{ytableau} 3 & 4 \\ 1 & 2 \end{ytableau}; \qquad \tau_1 \vee \tau_2 = \begin{ytableau} 3 & 4 \\ 2 & 2 & 3 & 3 & 4 \\ 1 & 1 & 1 & 1 & 2 \end{ytableau}.
\] 
If $\tau_1$ is of shape $\mu_1$ and $\tau_2$ is of shape $\mu_2$, then $\tau_1 \vee \tau_2$ is of shape $\mu_1+\mu_2$. 

Let $\nu \vdash n$, $\mu \vdash m$, and $\lambda \vdash nm$. We have the following lemma:

\begin{lem}\label{lemma1}
Let $\tau \in T(\lambda,\mu^n)$ such that $\overline{\Theta}_{\tau}: S^{\lambda} \rightarrow M^{\nu[\mu]}$ is non-zero. Then, let $\widetilde{\mu}$ be a partition of any integer $\widetilde{m}$, and let $\widetilde{\tau}$ be the only semistandard tableau of shape $(n\widetilde{m})$ and content $\widetilde{\mu}^n$. Then, the tableau $\widehat{\tau} = \tau \vee \widetilde{\tau}$ is such that $\overline{\Theta}_{\widehat{\tau}}: S^{\lambda + (n\widetilde{m})} \rightarrow M^{\nu[\mu + \widetilde{\mu}]}$ is also non-zero.
\begin{proof}
The morphism on a polytabloid $e(t)$ is
\[
\overline{\Theta}_{\tau}(e(t)) = \displaystyle \sum_{\substack{\tau' \sim_R \tau \\ \pi \in C_t}} \sgn(\pi) \phi(f_{t}(\pi \cdot \tau')).
\]
As the morphism is non-zero and $S^{\lambda}$ is a cyclic module, we have that for any bijective tableau $t$ of shape $\lambda$, $\overline{\Theta}_{\tau}(e(t)) \neq 0$. Fix one, and choose a plethystic tabloid $\{\bm{T}\}$ of shape $\nu[\mu]$ having a non-zero coefficient in $\overline{\Theta}_{\tau}(e(t))$.
Denote $B_{t}(\tau,\{\bm{T}\}) = \{ (\tau', \pi) \in T(\lambda,\mu^n) \times C_{\lambda} \ | \ \phi(f_t(\pi \cdot \tau')) = \{\bm{T}\} \}$, so that
\[
\overline{\Theta}_{\tau}(e(t)) = \displaystyle \sum_{(\tau',\pi) \in B_t(\tau,\{\bm{T}\})} \sgn(\pi) \{\bm{T}\} + \sum_{(\tau',\pi) \not\in B_t(\tau,\{\bm{T}\})} \sgn(\pi) \phi(f_{t}(\pi \cdot \tau'))
\]
This means that the coefficient of $\{\bm{T}\}$ in $\overline{\Theta}_{\tau}(e(t))$ is
\[
\displaystyle \sum_{(\tau',\pi) \in B_t(\tau,\{\bm{T}\})} \sgn(\pi),
\]
and by hypothesis, it does not vanish.

Now consider the morphism $\overline{\Theta}_{\widehat{\tau}}$. First, let $\widetilde{t}$ be the unique tableau of shape $(n\widetilde{m})$ and content $\{nm+1, \ldots, n(m+m')\}$ such that the row is increasing. Then, the tableau $\widehat{t} = t \vee \widetilde{t}$ is a bijective tableau of shape $\lambda + (n\widetilde{m})$ with entries in $\{1,2,\ldots,n(m+\widetilde{m})\}$. 

We want to describe the plethystic tabloid $\{\widehat{\bm{T}}\}$ of shape $\nu[\mu + \widetilde{\mu}]$ such that $\phi(f_{\widehat{t}}(\widehat{\tau})) = \{\widehat{\bm{T}}\}$. Consider the sets $A_k = \{t(c,1) \ | \ \tau(c,1) = k\}$. Then, we obtain $\{\widehat{\bm{T}}\}$ by adding the values of $A_k$ to the row of the inner tabloid of $\{\bm{T}\}$ containing the values $t(c,r)$ for $\tau(c,r)=k$. 

To prove the lemma, it is enough to show that the coefficient of $\{\widehat{\bm{T}}\}$ in $\overline{\Theta}_{\widehat{\tau}}(e( \ \widehat{t} \ )$ is also non-zero. Using the same technique and notations as before, this coefficient is
\[
\displaystyle \sum_{(\widehat{\tau}',\widehat{\pi}) \in B_{\widehat{t}}(\widehat{\tau},\{\widehat{\bm{T}}\})} \sgn(\widehat{\pi}).
\]
Consider $\widehat{\pi} \cdot \widehat{\tau}'$ for $(\widehat{\tau}',\widehat{\pi}) \in B_{\widehat{t}}(\widehat{\tau},\{\widehat{\bm{T}}\})$. Note that $C_{\widehat{t}}$ is the same as the inclusion of $C_t$ in $\s_{n(m+\widetilde{m})}$. So, by the description of $\{\widehat{\bm{T}}\}$, it must be the case that $\widehat{\pi} \cdot \widehat{\tau}' = (\pi \cdot \tau') \vee s$ for $(\tau',\pi) \in B_t(\tau, \{\bm{T}\})$ and $s$ a tableau of shape $(n\widetilde{m})$, for a unique tableau $s$. So, the coefficient of $\{\widehat{\bm{T}}\}$ in $\overline{\Theta}_{\widehat{\tau}}(e(\ \widehat{t} \ ))$ is the same as the coefficient of $\{\bm{T}\}$ in $\overline{\Theta}_{\tau}(e(t))$, and by assumption, it is non-zero. 
\end{proof}
\end{lem}

This lemma is sufficient to prove theorem \ref{thm1}.

\begin{proof}[Proof of theorem \ref{thm1}]
Let $\{\overline{\Theta}_{\tau_1}, ..., \overline{\Theta}_{\tau_r}\}$ be a basis of $\text{Hom}_{\s_{nm}}(S^{\lambda},M^{\nu[\mu]})$, so $a_{\nu[\mu]}^{\lambda} = r$. Denote $\widehat{\tau}_i = \tau_i \vee \widetilde{\tau}$ (as in lemma \ref{lemma1}). The theorem follows from the fact that $\{\overline{\Theta}_{\widehat{\tau}_1}, ..., \overline{\Theta}_{\widehat{\tau}_r}\}$ is linearly independant. Suppose that there exists $\alpha_1,...,\alpha_r$ such that $\alpha_1\overline{\Theta}_{\widehat{\tau}_1} + ... + \alpha_r\overline{\Theta}_{\widehat{\tau}_r} = 0$. It means that for any bijective tableau $\widehat{t}$ of shape $\lambda + (n\tilde{m})$,
\[
\displaystyle \sum_{i=1}^r \alpha_i \overline{\Theta}_{\widehat{\tau}_i}(\Delta_{\widehat{t}}) = 0.
\]
So, any term of this sum vanishes; in particular, for any plethystic tabloid $\{\widehat{\bm{T}}\}$ of shape $\nu[\mu + \widetilde{\mu}]$, its coefficient is 0.  But by lemma \ref{lemma1}, it means that the coefficient of $\{\bm{T}\}$, the plethystic tabloid of shape $\nu[\mu]$ obtained by removing entries from $nm+1$ to $n(m+\tilde{m})$, is also 0. As every tabloid of shape $\nu[\mu]$ is obtained this way, we have $\alpha_1\overline{\Theta}_{\tau_1} + ... + \alpha_r\overline{\Theta}_{\tau_r} = 0$. But this means that $\alpha_1 = ... =  \alpha_r = 0$, because this set is a basis. Hence, $\{\overline{\Theta}_{\widehat{\tau}_1}, ...,\overline{\Theta}_{\widehat{\tau}_r}\}$ must also be linearly independent.
\end{proof}

We also generalize another stability property, originally due to Dent \cite{Dent}. 

\begin{lem}
Let $\tau \in T(\lambda,\mu^n)$ such that $\overline{\Theta}_{\tau}: S^{\lambda} \rightarrow M_{\mu}^{\nu}$ is non-zero. Then, if the number of parts of $\mu$ is $\tilde{m}$, consider the tableau $\widetilde{\tau}$ of shape $(2^{n\widetilde{m}})$ such that the two entries of row $r$ are $r$. Then, $\widehat{\tau} = \widetilde{\tau} \vee \tau$ is such that $\overline{\Theta}_{\widehat{\tau}}: S^{\lambda + (2^{\tilde{m}n})} \rightarrow M^{\nu[\mu + (2^{\tilde{m}})]}$ is also non-zero.
\begin{proof}

We first claim that a tableau $\widehat{\tau}' \sim_R \widehat{\tau}$ doesn't have twice the same entry in a column if and only if $\widehat{\tau}' = \widetilde{\tau} \vee \tau'$ for $\tau' \sim_R \tau$ with the same property. For row $r > \ell(\lambda)$, there is two entries, and both are $r$. For row $r = \ell(\lambda)$, then all entries are bigger or equal to $r$. So, if the entries of the two columns added (the $\widetilde{\tau}$ part) are permuted, it is to be replaced by entries bigger than $r$. But every entry between $r$ and $n\widetilde{m}$ are already in cells above, so there would be twice the same entry in the same column. Making the same argument the row $\ell(\lambda)$ down to row $1$, we see that the $\widetilde{\tau}$ part must be preserved.

Consider the bijective tableau $\widetilde{t}: (2^{n\widetilde{m}}) \rightarrow \{nm+1,..., n(m+2\tilde{m})\}$ which is order-preserving (from $<_{\text{rlex}}$ to the order of $\N$). For $\widehat{t} = \widetilde{t} \vee t$, which is a bijective tableau of shape $\lambda + (2^{\tilde{m}n})$ with entries in $\{1,\ldots, n(m+2\widetilde{m})\}$, consider 
\[
\overline{\Theta}_{\widehat{\tau}}(e( \ \widehat{t} \ )) = \displaystyle \sum_{\substack{\widehat{\tau}' \sim_R \widehat{\tau} \\ \widehat{\pi} \in C_{\widehat{t}}}} \sgn(\widehat{\pi}) \phi(f_{\widehat{t}}(\widehat{\pi} \cdot \widehat{\tau}')).
\] 
We have that $C_{\widehat{t}} = C_t \times C_{\widetilde{t}}$. By the claim, we can then write:
\[
\overline{\Theta}_{\widehat{\tau}}(e( \ \widehat{t} \ )) = \displaystyle \sum_{\substack{\widehat{\tau}' \sim_R \widehat{\tau} \\ (\pi,\widetilde{\pi}) \in C_t \times C_{\widetilde{t}}}} \sgn(\pi)\sgn(\widetilde{\pi}) \phi(f_{\widehat{t}}((\widetilde{\pi} \cdot \widetilde{\tau}) \vee (\pi \cdot \tau)))
\] 
Now let’s take $\tau$ such that $\overline{\Theta}_{\tau} \neq 0$. We know that there exists a plethystic tabloid $\{\bm{T}\}$ of shape $\nu[\mu]$ such that its coefficient is non-zero. Using the same notation as in lemma \ref{lemma1}, its coefficient is
\[
\displaystyle \sum_{(\tau',\pi) \in B_{t}(\tau, \{\bm{T}\})} \sgn(\pi).
\]
Denote $\{\tilde{\bm{T}}\}$ the plethystic tabloid of shape $(2^{n\widetilde{m}})$ obtain by adding $nm+2k-1$ and $nm+2k$ to the row of the inner tabloid of $\{\bm{T}\}$ containing the values $t(c,r)$ for $\tau(c,r)=k$. This tableau is such that $\phi(f_{\widehat{t}}(\widehat{\tau})) = \{\widehat{\bm{T}}\}$. Its coefficient in $\overline{\Theta}_{\widehat{\tau}}(e( \ \widehat{t} \ ))$ is
\[
\displaystyle \sum_{(\widehat{\tau}',\widehat{\pi}) \in B_{\widehat{t}}(\widehat{\tau},\{\widehat{\bm{T}}\})} \sgn(\widehat{\pi})
\]
Consider $\widehat{\pi} \cdot \widehat{\tau}'$ for $(\widehat{\tau}',\widehat{\pi}) \in B_{\widehat{t}}(\widehat{\tau},\{\widehat{\bm{T}}\})$. By the claim, we can restrict ourselves to those of the form $(\widetilde{\pi} \cdot \widetilde{\tau}) \vee (\pi \cdot \tau')$ with $(\tau',\pi) \in B_t(\tau,\{\bm{T}\})$ and $\widetilde{\pi} \in C_{\widetilde{t}}$. But by construction of $\{\widehat{\bm{T}}\}$ from $\{\bm{T}\}$, it must be the case that $\widetilde{\pi}$ is unique, and is equivalent to a permutation of $\s_{n\widetilde{m}}$ for the values of the entries of $\widetilde{\tau}$. But to do so, $\widetilde{\pi}$ must permute the same number of cells in each column, hence is an even permutation. So, the coefficient of $\{\widehat{\bm{T}}\}$ in $\overline{\Theta}_{\widehat{\tau}}(e( \ \widehat{t} \ ))$ is
\[
\displaystyle \sum_{(\tau',\pi) \in B_{t}(\tau,\{\bm{T}\})} \sgn(\pi)\sgn(\widetilde{\pi}).
\]
As $\sgn(\widetilde{\pi}) = 1$, it is the same as the coefficient of $\{\bm{T}\}$ in $\overline{\Theta}_{\tau}(e(t))$, and by assumption, it is non-zero. 
\end{proof}
\end{lem}

This lemma is also sufficient to prove theorem \ref{thm2}, by using the same proof than the one for \ref{thm1}.

\section{Conclusion}

In the future, we hope to find more cases where there is an injection $M^{\nu[\mu]} \hookrightarrow M^{\mu[\nu]}$, \textrm{i.e.} more cases where $\nu \trianglelefteq \mu$. We also hope to find additional clues that this is an order relation, like propagation theorems, either for the plethysm coefficients or for the generalized Foulkes-Howe map. It would also be great to find a way to use efficiently the semistandard homomorphisms. 

In general, this conjecture deserved to be studied more. In addition to being a nice generalization of the Foulkes’ conjecture, it would be a great step toward the understanding of the plethysm of Schur functions.

\bibliographystyle{plain}
\bibliography{bibplethysm}

\end{document}